\documentclass[12pt,a4paper]{amsart}
\usepackage[all,cmtip]{xy}
 
\usepackage{enumerate, comment}
\usepackage{ amssymb, latexsym, amsmath}
\usepackage{fullpage}
\usepackage{url}
\usepackage{hyperref}
\usepackage{color}
\usepackage{helvet}
\usepackage{amsfonts}
\usepackage{tikz-cd}
\usepackage{amssymb}
\usepackage{amsthm}
\usepackage{amsmath}
\usepackage{amscd}
\usepackage[mathscr]{eucal}
\usepackage{indentfirst}
\usepackage{graphicx}
\usepackage{graphics}
\usepackage{pict2e}
\usepackage{epic}
\usepackage[OT2,T1]{fontenc}
\numberwithin{equation}{section}
\usepackage[margin=3.4cm]{geometry}

\newcommand{\Gal}{\operatorname{Gal}}

\theoremstyle{plain}

\newtheorem{Th}{Theorem}[section]
\newtheorem{Lemma}[Th]{Lemma}
\newtheorem{Cor}[Th]{Corollary}
\newtheorem{Prop}[Th]{Proposition}

 \theoremstyle{definition}

\newtheorem{Def}[Th]{Definition}

\newtheorem{?}[Th]{Problem}

\makeatletter
\newcommand*{\rom}[1]{\expandafter\@slowromancap\romannumeral #1@}
\makeatother
\newcommand{\Q}{\mathbb{Q}}

\newcommand{\Z}{\mathbb{Z}}

\newcommand{\Sel}{\operatorname{Sel}}
\newcommand{\GL}{\text{GL}}

\newcommand{\G}{\text{G}}

\newcommand{\K}{\Q(\mu_{p^{\infty}})}

\newcommand{\zp}{{\mathbb{Z}}_p}

\newcommand{\F}{\mathbb{F}}

\DeclareSymbolFont{cyrletters}{OT2}{wncyr}{m}{n}

\DeclareFontFamily{U}{wncy}{}
    \DeclareFontShape{U}{wncy}{m}{n}{<->wncyr10}{}
    \DeclareSymbolFont{mcy}{U}{wncy}{m}{n}
    \DeclareMathSymbol{\Sh}{\mathord}{mcy}{"58}

\title{Euler Characteristics and their Congruences in the Positive Rank Setting}
\author{Anwesh Ray}
\address{ Department of Mathematics\\Cornell University \\
  Malott Hall, Ithaca, NY 14853-4201 USA.}
  \email{ar2222@cornell.edu}
\author{R. Sujatha}
\address{Department of Mathematics \\ University of British Columbia \\
  Vancouver BC, V6T 1Z2, Canada.} 
  \email{sujatha@math.ubc.ca}

\begin{document}

\maketitle
\begin{abstract} 
The notion of the truncated Euler characteristic for Iwasawa modules is an extension of the notion of the usual Euler characteristic to the case when the homology groups are not finite. This article explores congruence relations between the truncated Euler characteristics for dual Selmer groups of elliptic curves with isomorphic residual representations, over admissible $p$-adic Lie extensions. Our results extend earlier congruence results from the case of elliptic curves with rank zero to the  case of higher rank elliptic curves. The results provide evidence for the $p$-adic Birch and Swinnerton-Dyer formula without assuming the main conjecture.
\end{abstract}
\section{Introduction}\label{section1}
Iwasawa theoretic invariants for modules arising from the Iwasawa theory of ordinary Galois representations
provide key insights into the arithmetic of such objects. Of particular interest is the behaviour of these invariants when one considers two ordinary Galois representations whose associated residual representations are isomorphic. Greenberg and Vatsal \cite{greenbergvatsal} initiated such a study for elliptic curves, and this was developed further in the works of Emerton-Pollack-Weston \cite{EPW}. Similar investigations were carried out in the context of non-commutative Iwasawa theory in, for instance, \cite{CSSLinks} and \cite{Dok}. 

In all the works referenced above, congruences between the corresponding $L$-values and Euler characteristics of dual Selmer groups of elliptic curves were established over special $p$-adic Lie extensions. A fundamental hypothesis when considering Euler characteristics was that the Euler characteristic was defined, which often entailed the assumption of finiteness of the dual Selmer group over the ground field. This article sets out to explore possible thematic generalizations of such congruence results when this finiteness hypothesis is removed. The natural substitute for the Euler characteristic is the truncated Euler characteristic as considered in \cite{Zerbes} for the cyclotomic extensions. This definition was generalized to  a broader class of  admissible $p$-adic Lie extensions in \cite{CSSLinks} and \cite{perrinriou}. It is striking that the congruence results extend to the rank one case, and to more general $p$-adic Lie-extensions. This leads us to believe that the truncated Euler characteristic is also an intrinsic arithmetic invariant in non-commutative Iwasawa theory.
\par This paper consists of five sections including this introduction. In section 2, we set up notation and requisite preliminaries. Section 3 proves the congruence results in the case of the cyclotomic $\zp$-extension and section 4 establishes similar congruence results in the setting of admissible, non-commutative $p$-adic Lie extensions. More specifically, we prove our results for false Tate curve extensions and $\GL_2$ extensions. In section 5, we discuss explicit numerical examples which demonstrate that our results are optimal.

\par Throughout, let $p\geq 5$ be a prime. Let $E_1$ and $E_2$ be two elliptic curves over the field $\Q$ of rational numbers. Let $\rho_i:\G_{\Q}\rightarrow \GL_2(\Z_p)$ be the Galois representation on the $p$-adic Tate module of $E_i$ and $E_i[p]$ denote the Galois module of the $p$-torsion points of the elliptic $E_i$. Denote by $\bar{\rho}_i:\G_{\Q}\rightarrow\GL_2(\F_p)$ the residual Galois representation. Let $N_i$ be the level of $E_i$ and set $N=N_1N_2$. Assume that
\begin{enumerate}
    \item\label{one} $E_1$ and $E_2$ both have the same algebraic rank $g$,
     \item\label{two} the $p$-adic Galois representations
     $\rho_1\equiv \rho_2\mod{p}$,
     \item\label{three} the residual Galois representations $\bar{\rho}_i$ are irreducible,
    \item\label{four} $E_1$ and $E_2$ both have good ordinary reduction at $p$,
    \item\label{five} $\Sh(E_i/\Q)[p]$ is finite for $i=1,2$,
    \item\label{six} the $p$-adic height pairings on $E_1$ and $E_2$ are non-degenerate.
\end{enumerate}
Conditions $(\ref{five})$ and $(\ref{six})$ are always expected to be true. Let $\Q^{cyc}$ denote the cyclotomic $\Z_p$-extension of $\Q$ and put $\Gamma:=\text{Gal}(\Q^{cyc}/\Q)$. Let $\Q_{\infty}/\Q$ be an admissible $p$-adic Lie-extension of $\Q$ and $G:=\text{Gal}(\Q_{\infty}/\Q)$. Let $\chi_t(\Gamma, E_i)$ (resp. $\chi_t(G, E_i)$) denote the truncated Euler characteristic of $E_i$ with respect to $\Gamma$ (resp. $G$), see section 2 for the definition. 
\par The main conjecture of Iwasawa theory relates the algebraic invariants attached to the dual Selmer group of an elliptic curve with the $p$-adic L-function, which interpolates values of the complex L-functions. Vatsal \cite{VatsalCP} and Greenberg-Vatsal \cite{greenbergvatsal} study congruence properties for complex $L$-values as well as for the $p$-adic L-functions attached to $E_1$ and $E_2$. On the algebraic side, the main conjecture in conjunction with work of Schneider \cite{heightpairings2} and Perrin-Riou would then predict congruences for the Euler characteristic of the Selmer group.
\par For an elliptic curve $E$ with algebraic rank zero, denote the Euler characteristic of the Selmer group of $E$ over the cyclotomic $\Z_p$-extension by $\chi(\Gamma, E)$ (cf. \cite[chapter 3]{GCEC}). When $E$ has algebraic rank zero, the Euler-characteristic $\chi(\Gamma, E)$ coincides with the truncated Euler-characteristic $\chi_t(\Gamma,E)$. Shekhar and the second author \cite{shekharsujatha} deduced congruence results for the Euler characteristic of pairs of congruent elliptic curves $E_1$ and $E_2$ of rank zero. In particular, they deduce that if $\chi(\Gamma, E_1)=1$, then $\chi(\Gamma, E_2)=1$. Let $m$ be a $p$-power free integer coprime to $p$ as well as to the conductors of both elliptic curves. For the false Tate curve extension defined by $\Q_{\infty}=\Q(\mu_{p^{\infty}}, m^{\frac{1}{p^{\infty}}})$ it is shown in \cite[Theorem 3.4]{shekharsujatha} if $\chi(G, E_1)=1$ then so is $\chi(G,E_2)=1$. In this article, analogous results for truncated Euler characteristics in the positive rank setting are proved. We stress that our methods are different from those in \cite{shekharsujatha}, and hence yield another proof of the results in the rank zero case as well. Generalizations of the results in this paper to modular forms and abelian varieties over arbitrary number fields are currently being investigated.
\par The $p$-adic version of the Birch and Swinnerton-Dyer conjecture has been studied, for instance in \cite{padicBSD},\cite{perrinriou} and \cite{heightpairings2}. The results in this paper for the truncated Euler characteristic yield interesting consequences for the $p$-adic Birch and Swinnerton-Dyer conjecture. Let $E$ be an elliptic curve over $\Q$ with good ordinary reduction at $p$ for which $E[p]$ is irreducible as a Galois module. The $p$-adic $L$-function, denoted by $\mathcal{L}(E/\Q,T)$, is a power series in $\Z_p[[T]]$. The characteristic series of the dual Selmer group of $E$ is denoted by $f_E^{alg}(T)$. The main conjecture predicts that $\mathcal{L}(E/\Q,T)$ coincides with $f_E^{alg}(T)$ up to a unit in $\Z_p[[T]]$. The $p$-adic Birch and Swinnerton-Dyer conjecture predicts that the order of vanishing of $\mathcal{L}(E/\Q,T)$ at $T=0$ is equal to the rank of $E(\Q)$ and that there is an exact formula for the leading coefficient of $\mathcal{L}(E/\Q,T)$. The leading coefficient is predicted to be equal to $p^{-g}\rho_p(E)$, where \[\rho_p(E):=\frac{R_p(E/\Q)\times \#(\Sh (E/\Q)[p])}{\#(E(\Q)[p])^2}\times \tau(E)\times \#\tilde{E}(\F_p)^2.\] (see Proposition $\ref{PerrinRiouresult}$). In the above formula, $R_p(E/\Q)$ is the $p$-adic regulator of $E$, defined as the determinant of the $p$-adic height pairing studied in \cite{perrinriou}, \cite{heightpairings1} and \cite{heightpairings2}. When the rank of $E(\Q)$ is equal to zero, the $p$-adic regulator is set to be equal to $1$. When the rank of $E(\Q)$ is positive, the $p$-adic regulator is conjectured to be nonzero. The term $\mathcal{\tau}(E):=\prod_l c_l$ where $c_l$ is the index of $E_0(\Q_l)$ in $E(\Q_l)$,  with $E_0(\Q_l)$ being the subgroup of $l$-adic points with non-singular reduction modulo $l$. The term $\tilde{E}$ is the reduced curve at $p$.
\par It is shown by Perrin-Riou and Schneider that if the $p$-adic regulator of $E$ is not zero, then the order of vanishing of $f_{E}^{alg}(T)$ is equal to the rank of $E(\Q)$ and the leading term of $f_{E}^{alg}(T)$ is equal to $p^{-g}\rho_p(E)$, up to a $p$-adic unit. As a consequence, the $p$-adic Birch and Swinnerton-Dyer conjecture for $E$ follows from the main conjecture for $E$. This is where the truncated Euler characteristic intervenes. The truncated Euler-characteristic $\chi_t(\Gamma, E)$ is equal to the leading coefficient of $f_E^{alg}(T)$ up to a unit (see \cite[Lemma 2.11]{Zerbes}). As a consequence of the results of Schneider and Perrin-Riou mentioned above, $\chi_t(\Gamma, E)=p^{-g}\rho_p(E)$.
\par Let $E_1$ and $E_2$ be elliptic curves satisfying conditions $(\ref{one})$ to $(\ref{six})$. In particular, $E_1[p]$ and $E_2[p]$ are isomorphic as Galois representations. We show that there is an explicit relationship between $\chi_t(\Gamma, E_1)$ and $\chi_t(\Gamma, E_2)$ (see Theorem $\ref{gammacongruence}$) which further relates $p^{-g}\rho_p(E_1)$ and $p^{-g}\rho_p(E_2)$ (see Corollory $\ref{BSDmaincor}$). Next we consider the coefficient of $T^g$ of the $p$-adic L-function $\mathcal{L}(E_i/\Q,T)$ which is given by \[\mathcal{L}^{(g)}(E_i/\Q,0):=\frac{1}{g!}\left(\frac{d^g}{dT^g}\mathcal{L}(E_i/\Q,T)\right)_{\restriction T=0}.\] The $p$-adic Birch and Swinnerton-Dyer conjecture predicts that $\mathcal{L}^{(g)}(E_i/\Q,0)=p^{-g}\rho_p(E_i)$. Using the results of Perrin-Riou and Schneider, which related the latter term $p^{-g}\rho_p(E_i)$ to the truncated Euler-characteristic, our results would translate under the $p$-adic Birch Swinnerton-Dyer conjecture, to a congruence relation between $\mathcal{L}^{(g)}(E_1/\Q,0)$ and $\mathcal{L}^{(g)}(E_2/\Q,0)$. However, we show that this relationship is indeed satisfied independent of the $p$-adic Birch and Swinnerton-Dyer conjecture (see Theorem $\ref{corolloryLfunctions}$), thereby providing evidence for the conjecture.
\par We provide a sketch of the methods used in this manuscript. The set of primes $\Sigma_0$ is chosen to consist of exactly those primes at which $E_1$ or $E_2$ have bad reduction. One may choose a larger set, however, results are optimal for this set. Greenberg and Vatsal in \cite{greenbergvatsal} consider the $\Sigma_0$-imprimitive Selmer group over the cyclotomic $\Z_p$-extension, which we denote by $\Sel^{\Sigma_0}(E/\Q^{cyc})$ (see section $\ref{section2}$ for the definition). Set $f_{E_i,\Sigma_0}^{alg}(T)$ for the characteristic series of $\Sel^{\Sigma_0}(E_i/\Q^{cyc})$ and denote by $\mu_{E_i,\Sigma_0}^{alg}$ and $\lambda_{E_i,\Sigma_0}^{alg}$ the $\mu$ and $\lambda$-invariants of $f_{E_i,\Sigma_0}^{alg}(T)$ respectively. Greenberg and Vatsal show that
\[\mu_{E_1,\Sigma_0}^{alg}=0\Leftrightarrow\mu_{E_2,\Sigma_0}^{alg}=0\text{, and if }\mu_{E_1,\Sigma_0}^{alg}=\mu_{E_2,\Sigma_0}^{alg}=0\text{, then }\lambda_{E_1,\Sigma_0}^{alg}=\lambda_{E_2,\Sigma_0}^{alg}\](see the discussion on \cite[pp. 43]{greenbergvatsal} preceding Remark (2.10)). Using these relations, it becomes possible to deduce that the coefficient of $T^g$ in $f_{E_1,\Sigma_0}(T)$ is $p$-adic unit if and only if that of $T^g$ in $f_{E_2,\Sigma_0}(T)$ is a $p$-adic unit (see the proofs of Corollory $\ref{prethmcorollory}$ and Theorem $\ref{gammacongruence}$). By the results of Schneider and Perrin-Riou discussed earlier, the order of vanishing of $f_{E_i}^{alg}(T)$ at $T=0$ is $g$, and its leading coefficient is equal to $\chi_t(\Gamma, E_i)$, up to a $p$-adic unit. Setting $\Phi_{E_i}:=\prod_{l\in \Sigma_0} |L_l(E_i,1)|_p$, we then compare $\Phi_{E_i}\times \chi_t(\Gamma, E_i)$ for the two elliptic curves.
\par Our results are proved independent of the main conjecture. Note that the main conjecture is known in this setting for $E_i$ if there is a prime $q\neq p$ such that $q||N_i$ and the residual representation $\bar{\rho}_{E_i,p}$ on $E_i[p]$ is ramified at $q$. This is the celebrated theorem of Skinner and Urban \cite{skinnerurban}. We stress that this additional condition is not imposed in proving our results. Leveraging results proved for the cyclotomic extension allows us to prove analogous results for other $p$-adic Lie-extensions (Theorem $\ref{Gcongruence}$ for the false-Tate extension and Theorem $\ref{GEcongruence}$ for certain $\GL_2$-extensions).
\\
\newline\textit{Acknowledgements:} The authors would like to thank Ravi Ramakrishna, Sudhanshu Shekhar and Christian Wuthrich for helpful discussions. The second author gratefully acknowledges support from NSERC Discovery grant 2019-03987.
\section{Preliminaries}\label{section2}
\par Let $p$ be a prime number and assume that $p\geq 5$. Let $E$ be an elliptic curve over $\Q$ with good ordinary reduction at the prime $p$ and $E_{p^{\infty}}$ denote the Galois-module of $p$-power division points on $E$. Let $\Sigma$ be a finite set of primes containing $p$ and the primes at which $E$ has bad reduction.
\subsection{Selmer groups and their Characteristic Polynomials}Denote by $\Q_{\Sigma}$ the maximal extension of $\Q$ in which all primes $l\notin \Sigma$ are unramified, and denote by $\G_{\Q,\Sigma}$ the Galois group of $\Q_{\Sigma}$ over $\Q$. Let $L \subset \Q_{\Sigma}$ be a subfield, we shall be concerned with the cyclotomic $\Z_p$-extension or an admissible $p$-adic Lie extension. Set $\G_{\Sigma}(L):=\text{Gal}(\Q_{\Sigma}/L)$, the Selmer group $\Sel(E/L):=\text{ker} \lambda_{\Sigma}(L)$. Here $\lambda_{\Sigma}(L)$ is the localization map
\[\lambda_{\Sigma}(L):H^1\left(\G_{\Sigma}(L),E_{p^{\infty}}\right)\rightarrow \bigoplus_{l\in \Sigma} \mathcal{H}_l(L),\] where $\mathcal{H}_l(L)$ is taken as in \cite[pp. 23]{greenbergvatsal}.
Let $\Sigma_0\subseteq \Sigma$ be a subset of primes which contains the primes at which $E$ has bad reduction and does not contain $p$. In section $\ref{section3}$, we shall specialize $\Sigma_0$ further. The $\Sigma_0$-imprimitive (or non-primitive) Selmer group $\Sel^{\Sigma_0}(E/L)$ is the kernel of the localization map 
\[\lambda_{\Sigma}^{\Sigma_0}(L):H^1\left(G_{\Sigma}(L),E_{p^{\infty}}\right)\rightarrow \bigoplus_{l\in \Sigma\backslash \Sigma_0} \mathcal{H}_l(L).\]Let $G$ be a profinite group, the Iwasawa algebra $\Lambda(G)$ is defined as
\[\Lambda(G):=\varprojlim_{U}\Z_p[G/U]\]where the inverse limit is taken with respect to open subgroups of $G$. Put $\Gamma=\Gal(\Q^{cyc}/\Q)$ and set $\Lambda$ to be equal to the Iwasawa algebra $\Lambda(\Gamma)$. Note that for any extension $L/\Q$, finite or infinite, the Selmer groups $\Sel(E/L)$ and $\Sel^{\Sigma_0}(E/L)$ are discrete modules over the Iwasawa algebra $\Lambda(\Gal(L/\Q))$. Denote by
\[X(E/L):=\text{Hom}\left(\Sel(E/L),\Q_p/\Z_p\right)\]
\[X^{\Sigma_0}(E/L):=\text{Hom}\left(\Sel^{\Sigma_0}(E/L),\Q_p/\Z_p\right)\]the compact Pontrjagin duals of the Selmer and $\Sigma_0$-imprimitive Selmer groups respectively. It follows from a deep theorem of Kato \cite{kato} that the dual Selmer group $X(E/\Q^{cyc})$ is a torsion $\Lambda$-module. This is equivalent to the localization map $\lambda_{\Sigma}(\Q^{cyc})$ being surjective (cf. \cite[Lemma 2.1]{CSSLinks}). It is an easy result that $X(E/\Q^{cyc})$ is a finitely generated $\Lambda$-module. By the classification theorem of finitely generated torsion $\Lambda$-modules, there is a pseudo-isomorphism 
\[X(E/\Q^{cyc})\sim \left(\bigoplus_{i=1}^n \Lambda/(f_i(T)^{a_i})\right)\oplus \left(\bigoplus_{j=1}^m \Lambda/p^{\mu_j} \right).\]Here, one identifies $\Lambda$ with the formal power series ring $\Z_p[[T]]$ after making a choice of a topological generator $\gamma$ of $\Gamma$ and letting $T:=\gamma-1$. The $f_i(T)$ are irreducible monic polynomials all of whose non-leading coefficients are divisible by $p$. Such polynomials are called distinguished polynomials. The algebraic Iwasawa invariants are defined by 
\[\lambda_E^{alg}=\sum_{i=1}^n a_i \operatorname{deg} f_i\text{, and } \mu_E^{alg} =\sum_{j=1}^m \mu_j.\]
The characteristic polynomial of $X(E/\Q^{cyc})$ is defined by
\[f_E^{alg}(T)=p^{\mu_E^{alg}}\prod_{i=1}^n f_i(T)^{a_i}.\]It is the unique element generating the characteristic ideal of $X(E/\Q^{cyc})$ which is a power of $p$ times a distinguished polynomial. It is expected that when $E$ is an elliptic curve for which $E[p]$ is irreducible, then $\mu_{E}^{alg}=0$. It is well known that the condition $\mu_{E}^{alg}=0$ depends only on the residual representation $E[p]$ (see \cite[pp. 19]{greenbergvatsal}). That is, if two elliptic curves $E_1$ and $E_2$ over $\Q$ with good ordinary reduction at $p$ have isomorphic residual representation at $p$, then $\mu_{E_1}^{alg}=0$ if and only if $\mu_{E_2}^{alg}=0$. The $\lambda$-invariant $\lambda_{E}^{alg}$ need not be zero and moreover, need not depend on the residual representation (see \cite[pp. 22]{greenbergvatsal}). We define Iwasawa invariants $\mu_{E,\Sigma_0}^{alg}$ and $\lambda_{E,\Sigma_0}^{alg}$ associated to the $\Sigma_0$-imprimitive dual Selmer group $X^{\Sigma_0}(E/\Q^{cyc})$. For $l\in \Sigma_0$, let $h_E^{(l)}(T)$ be the characteristic polynomial of the Pontryajin dual of the $\mathcal{H}_l(\Q^{cyc})$. Let $\sigma_E^{(l)}$ be such that $\mathcal{H}_l(\Q^{cyc})\simeq (\Q_p/\Z_p)^{\sigma_{E}^{(l)}}$ as in \cite[pp. 23]{greenbergvatsal}. The characteristic polynomial of the $\Lambda$-module $\Sel^{\Sigma_0}(E/\Q^{cyc})$ is denoted by $f_{E,\Sigma_0}^{alg}(T)$. There is a short exact sequence
\[0\rightarrow \Sel(E/\Q^{cyc})\rightarrow \Sel^{\Sigma_0}(E/\Q^{cyc})\rightarrow \prod_{l\in \Sigma_0} \mathcal{H}_l(\Q^{cyc})\rightarrow 0.\]As a result,
\[f_{E,\Sigma_0}^{alg}(T)=f_{E}^{alg}(T)\prod_{l\in \Sigma_0} h_E^{(l)}(T).\] Letting $\lambda_{E,\Sigma_0}^{alg}$ denote the $\lambda$-invariant of $X^{\Sigma_0}(E/\Q^{cyc})$. It is related to $\lambda_{E}^{alg}$ according to the formula \cite[equation (7)]{greenbergvatsal}
\[\lambda_{E,\Sigma_0}^{alg}=\lambda_{E}^{alg}+\sum_{l\in \Sigma_0} \sigma_E^{(l)}.\]
\subsection{$p$-adic L-functions}
\par There is a similar story for $p$-adic L-functions. Let $g$ denote the algebraic rank of $E$ and $\Omega_{E}$ denote the the real N\'eron period of $E$. For any even Dirichlet character $\rho$, it is well known that $L(E/\Q, \rho, 1)/\Omega_{E}\in\bar{\Q}_p$. Recall that in the previous subsection, a topological generator $\gamma\in\Gamma$ has been fixed. Let $\mu_{p^{\infty}}$ denote the Galois module of $p$-torsion roots of unity. Given $\xi\in \mu_{p^{\infty}}$, let $m$ be the integer such that the order of $\xi$ is $p^{m-1}$. Associate to $\xi$ a character $\rho:\Gamma\rightarrow \mu_{p^{\infty}}$ of finite order defined by $\rho(\gamma)=\xi$. Let $\alpha_p\in \bar{\Q}_p$ denote the unit root of the characteristic polynomial $1-a_p(E)X+pX^2$. Mazur and Swinnerton-Dyer associate to $E$ an element $\mathcal{L}(E/\Q, T)\in \Lambda\otimes \Q_p$ satisfying the property
\[\mathcal{L}(E/\Q, \xi-1)=\tau(\rho^{-1})\cdot \alpha_p^{-m} \frac{L(E/\Q, \rho,1)}{\Omega_E}.\] Here, $\tau(\rho^{-1})$ denotes the Gauss-sum of the character $\rho^{-1}$. Let $\chi$ be the $p$-adic cyclotomic character and
let $\kappa=\chi_{\restriction \Gamma}:\Gamma\xrightarrow{\sim} (1+p\Z_p)$. The $p$-adic $L$-function $L_p(E,s)$ is defined by
\[L_p(E,s)=\mathcal{L}(E/\Q, \kappa(\gamma)^{1-s}-1).\] Set \[L_p^{(g)}(E,s):=\frac{1}{g!}\left(\frac{d}{ds}\right)^g L_p(E,s),\] likewise, set \[\mathcal{L}^{(g)}(E/\Q,T):=\frac{1}{g!}\left(\frac{d}{dT}\right)^g \mathcal{L}^{(g)}(E/\Q,T).\]
\par
 Let $P_l(E/\Q, \rho,T)$ be the local-factor defined as follows
\[P_l(E/\Q,\rho, T)=\begin{cases}
1-a_l \rho(l) T+l\rho(l)^2 T^2\text{ if }l \text{ is a prime of good reduction,}\\
1-a_l \rho(l) T\text{ if }E \text{ has bad reduction at }l.
\end{cases}\]
If $l$ is a prime of bad reduction then
\[a_l=\begin{cases}
1\text{ if }E\text{ has split multiplicative reduction},\\
-1\text{ if }E\text{ has non-split multiplicative reduction},\\
0\text{ if }E\text{ has additive reduction.}
\end{cases}\]
When $\rho$ is the trivial character it is dropped from the notation. The local Euler factor $L_l(E,s)$ at $l$ is equal to $L_l(E,s):=P_l(E,l^{-s})^{-1}$.
Let $\mathcal{P}_l(E/\Q,T)\in \Lambda$ be defined by the relation
$\mathcal{P}_l(E/\Q,\xi-1)=P_l(E/\Q,\rho,l^{-1})$. The non-primitive $L$-function is \[L^{\Sigma_0}(E/\Q,\rho,s)=L(E/\Q, \rho,s)\prod_{l\in \Sigma_0}P_l(E/\Q,\rho, l^{-s}).\]
The function $\mathcal{L}(E/\Q,T)$ is modified to a non-primitive version $\mathcal{L}^{\Sigma_0}(E/\Q, T)$ for which 
\[\mathcal{L}^{\Sigma_0}(E/\Q, \xi-1)=\tau(\rho^{-1})\alpha_p^{-m}\frac{L^{\Sigma_0}(E/\Q,\rho,1)}{\Omega_E}\]
(cf. \cite{greenbergvatsal} and \cite{MTT}).
For $l\in \Sigma_0$, let $\gamma_l$ denote the Frobenius automorphism of $l$ in $\Gamma=\Gal(\Q_{\infty}/\Q)$. Factor $P_l(E/\Q,T)=(1-\alpha_l T)(1-\beta_l T)$. Consider the power series $\mathcal{P}_l(E/\Q,T)\in \Z_p[[T]]$ corresponding to \[\mathcal{P}_l:=\left(1-\alpha_l l^{-1}\gamma_l\right) \left(1-\beta_l l^{-1}\gamma_l\right)\in \Lambda.\]Observe that $\mathcal{P}_l(E/\Q,0)=P_l(E/\Q, l^{-1})=L_l(E,1)^{-1}$. The following relationship is satisfied
\begin{equation}\label{equationscriptL}\mathcal{L}^{\Sigma_0}(E/\Q,T)=\mathcal{L}(E/\Q,T)\prod_{l\in \Sigma_0} \mathcal{P}_l(E/\Q,T).\end{equation}
It is well known that if $E[p]$ is irreducible as a Galois module, then $\mathcal{L}(E,T)\in \Lambda$ (see \cite[Proposition 3.7]{greenbergvatsal}). By the Weierstrass Preparation theorem, 
\[\mathcal{L}(E/\Q, T)=p^{\mu}\cdot u(T)g(T)\] where $u(T)$ is a unit and $g(T)$ a distinguished polynomial. Let $\mu_{E}^{anal}:=\mu$ as above and $f_{E}^{anal}(T):=p^{\mu_{E}^{anal}}g(T)$. Let $\lambda_{E}^{anal}$ denote the degree of $f_{E}^{anal}(T)$. The main conjecture states that $f_{E}^{alg}(T)=f_{E}^{anal}(T)$. Under additional hypotheses, this has been proved by Skinner and Urban \cite{skinnerurban}. In this manuscript we do not assume these additional hypotheses and the results are independent of the celebrated theorem of Skinner-Urban.
\par One has analogous definitions for analytic invariants associated to the $\Sigma_0$-imprimitive $p$-adic L-function $\mathcal{L}^{\Sigma_0}(E,T)$. For $l\in \Sigma_0$, it is the case that $\mathcal{P}_l$ generates the characteristic ideal of the Pontryajin dual of $\mathcal{H}_l(\Q^{cyc})$ \cite[Proposition 2.4]{greenbergvatsal}. It follows from the relation $\ref{equationscriptL}$ that \[\lambda_{E,\Sigma_0}^{anal}=\lambda_{E}^{anal}+\sum_{l\in \Sigma_0} \sigma_E^{(l)}.\] Greenberg and Vatsal \cite[Theorem 1.4]{greenbergvatsal} show that if $E_1$ and $E_2$ are two elliptic curves defined over $\Q$ which are congruent modulo $p$, then if the equalities 
\[\mu_{E_1}^{alg}=\mu_{E_1}^{anal}=0\text{, and } 
\lambda_{E_1}^{alg}=\lambda_{E_1}^{anal}\]hold, then so do the equalities\[\mu_{E_2}^{alg}=\mu_{E_2}^{anal}=0\]and 
\[\lambda_{E_2}^{alg}=\lambda_{E_2}^{anal}.\]Kato showed that the polynomial $f_{E}^{alg}(T)$ always divides $f_E^{anal}(T)$. As a result, if the main conjecture is true for $E_1$ and the $\mu$-invariant for $E_1$ is zero, then the same is true for $E_2$.
\subsection{Generalized Euler Characteristics} Recall that an admissible $p$-adic Lie-extension $\Q_{\infty}/\Q$ is a Galois extension for which
\begin{itemize}
    \item $G=\text{Gal}(\Q_{\infty}/\Q)$ is isomorphic to a compact $p$-adic Lie group,
    \item all but finitely many primes are unramified in $\Q_{\infty}$,
    \item $\Q^{\text{cyc}}\subset \Q_{\infty}$,
    \item $G$ does not contain an element of order $p$.
\end{itemize} Set $H=\Gal(\Q_{\infty}/\Q^{\text{cyc}})$ and note that $G/H\simeq \Gamma=\Gal(\Q^{\text{cyc}}/\Q)$. The Euler characteristic $\chi(\Gamma,E):=\chi(\Gamma, \Sel(E/\Q^{\text{cyc}}))$. When $E$ has positive rank, the Euler characteristic is not defined. The natural analog is the truncated Euler characteristic which we define below. For any homomorphism of abelian groups $f:M_1\rightarrow M_2$, denote by $\operatorname{ker} f$ and $\operatorname{cok} f$ the kernel of $f$ and the cokernel of $f$ respectively.
\begin{Def}
Let $D$ be a discrete $p$-primary $\Gamma$-module let $\phi_D$ be the natural map 
\[\phi_D: H^0\left(\Gamma, D\right)=D^{\Gamma}\rightarrow D_{\Gamma}\simeq H^1\left(\Gamma,D\right)\]
for which $\phi_D(x)$ is the residue class of $x$ in $D_{\Gamma}$. If both $\operatorname{ker}(\phi_D)$ and $\operatorname{cok}(\phi_D)$ are finite $D$ is said to have finite truncated Euler-characteristic. In this case, the truncated Euler characteristic is defined by \[\chi_t\left(\Gamma, D\right):=\frac{\#\text{ker}(\phi_D)}{\#\text{cok}(\phi_D)}.\]
\end{Def}
The truncated $G$-Euler characteristic is a variation of the above definition.
\begin{Def}Let $D$ be a discrete $p$-primary $G$-module, define 
\[\psi_D: H^0(\Gamma,D^H)\rightarrow H^1(G,D)\] as the composite of the natural map \[\phi_{D^H}:H^0(\Gamma, D^H)\rightarrow H^1(\Gamma, D^H)\] with the inflation map $H^1(\Gamma, D^H)\rightarrow H^1(G,D)$. The module $D$ has \textit{finite truncated $G$-Euler characteristic} if both $\operatorname{ker}(\psi_D)$ and $\operatorname{cok}(\psi_D)$ are finite and the truncated $G$-Euler characteristic is then defined by
\[\chi_t(G,D):=\frac{\# \operatorname{ker}(\psi_D)}{\# \operatorname{cok}(\psi_D)}.\]
\end{Def}
Denote by 
\[\begin{split}
    &\chi_t(\Gamma, E)=\chi_t\left(\Gamma, \Sel(E,\Q^{\text{cyc}})\right)\\
     &\chi_t(G, E)=\chi_t\left(G, \Sel(E,\Q_{\infty})\right).
\end{split}\]The truncated Euler characteristic $\chi_t(\Gamma, E)$ is, up to a $p$-adic unit, the leading term of the $p$-adic L-function $L_p(E,s)$ (cf. Theorem $\ref{TheoremchiLfunctions}$) and is therefore a $p$-adic integer. Let $\tilde{E}$ denote the reduced curve at $p$. For a prime $l$, we let $c_l(E):=\#(E(\Q_l)/E_0(\Q_l))$ and set $\tau(E):=\prod_l c_l $. Two $p$-adic integers $a\sim b$ if $a$ and $b$ are not zero and $a/b$ is a $p$-adic unit in $\Z_p$. In the rank zero case, if $\Sel(E/\Q)$ is finite, the Euler characteristic $\chi(\Gamma, E)$ is related to the Birch Swinnerton-Dyer exact formula as follows
\begin{equation}\label{BSDEulerchar}\chi(\Gamma, E)\sim \frac{\#\left(\Sh(E/\Q)[p]\right)}{\#\left(E(\Q)[p]\right)^2}\times \tau(E) \times  \#\tilde{E}(\F_p)^2.\end{equation}
\par In the positive rank case, the $p$-adic height pairing (cf. \cite{heightpairings1} and \cite{heightpairings2}) plays a role. For an elliptic curve $E$ over $\Q$ with positive rank, the $p$-adic height pairing is conjectured to be non-degenerate (cf. \cite{heightpairings2}). The $p$-adic regulator $R_p(E/\Q)$ is the determinant of the $p$-adic height pairing.
\begin{Prop}\label{PerrinRiouresult} (cf. \cite[section 3]{CSSLinks} and \cite{perrinriou})
Assume that $\Sh(E/\Q)[p]$ is finite, that the residual Galois representation on $E[p]$ is irreducible, that $E$ has good ordinary reduction at $p$ and that $R_p(E/\Q)$ is nonzero. Then
$\chi_t(\Gamma, E)\sim p^{-g}\rho_p(E)$
where 
\[\rho_p(E):=\frac{R_p(E/\Q)\times \#(\Sh (E/\Q)[p])}{\#(E(\Q)[p])^2}\times \tau(E)\times \#\tilde{E}(\F_p)^2.\]
\end{Prop}
If the order of $\Sh(E/\Q)[p]$ is known to be finite and the $p$-adic regulator of $E$ is non-zero, then it is known that the order of vanishing of $f_{E}^{alg}(T)$ at $T=0$ is equal to $g$, the algebraic rank of $E$. This is a result of Schneider and Perrin-Riou (see for instance \cite[Theorem 2$'$]{heightpairings2}). These conditions are imposed on the elliptic curves $E_1$ and $E_2$. The truncated Euler characteristic is related to the characteristic series $f_{E}^{alg}(T)$. Write $f_{E}^{alg}(T)=T^g g_{E}(T)$ where $g_{E}(0)\neq 0$. The following is a direct consequence of \cite[Lemma 2.11]{Zerbes}.
\begin{Lemma}\label{truncatedECSelmer1}
Assume that $E$ is an elliptic curve over $\Q$ for which
\begin{enumerate}
    \item $E$ has good ordinary reduction at $p$,
    \item $E[p]$ is irreducible as a Galois module,
    \item $\Sh(E/\Q)[p]$ is finite,
    \item the $p$-adic regulator of $E$ is non-zero.
\end{enumerate}
Then, $\chi_t(\Gamma, E)=|g_E(0)|_p^{-1}$.

In particular if $E$ is an elliptic curve satisfying the above conditions, it follows that $\chi_t(\Gamma, E)=p^N$ for some $N\in \Z_{\geq 0}$.
\end{Lemma}
\begin{Def}\label{PhiDef}Let $E$ be an elliptic curve over $\Q$ and $\Sigma_0$ a set of primes containing the primes at which $E$ has bad reduction and not containing $p$. Set $\Phi_{E}:=\prod_{l\in \Sigma_0} |L_l(E, 1)|_p$. Here, $|L_l(E, 1)|_p$ is set to be equal to $0$ if $L_l(E, 1)^{-1}=0$.
\end{Def}
Writing $f_{E,\Sigma_0}^{alg}(T)=T^g g_{E,\Sigma_0}(T)$, the following is a direct consequence of Lemma $\ref{truncatedECSelmer1}$.
\begin{Lemma}
Let $E$ be an elliptic curve satisfying the conditions of Lemma $\ref{truncatedECSelmer1}$. Then 
\[\Phi_E\times \chi_t(\Gamma, E)=|g_{E,\Sigma_0}(0)|_{p}^{-1}.\]
\end{Lemma}
In the above $|g_{E,\Sigma_0}(0)|_{p}^{-1}$ is taken to be $0$ if $g_{E,\Sigma_0}(0)=0$.
\begin{Cor}\label{prethmcorollory}
Let $E$ be an elliptic curve satisfying the conditions of Lemma $\ref{truncatedECSelmer1}$. Then $\Phi_E\times \chi_t(\Gamma, E)=1$ if and only if $\mu_{E}^{alg}=0$ and $\lambda_{E,\Sigma_0}=g$. 
\end{Cor}
\begin{proof}
Suppose that $\Phi_E\times \chi_t(\Gamma, E)=1$, then $g_{E,\Sigma_0}(T)$ is a unit. As a result, $f_{E,\Sigma_0}^{alg}(T)$ is $T^g$ times a unit. It follows that $\mu_{E,\Sigma_0}^{alg}=0$ and that $f_{E,\Sigma_0}^{alg}(T)$ is a distinguished polynomial. As a result, $f_{E,\Sigma_0}^{alg}(T)=T^g$ and in particular, $\lambda_{E,\Sigma_0}=g$. Finally, we point out that since each factor $\mathcal{P}_l$ is not divisible by $p$ and consequently, $\mu_{E}^{alg}=\mu_{E,\Sigma_0}^{alg}=0$. 
\par Conversely, suppose that $\mu_{E}^{alg}=0$ and that $\lambda_{E,\Sigma_0}^{alg}=g$. Since $f_{E,\Sigma_0}^{alg}(T)$ is divisible by $T^g$ it follows that $f_{E,\Sigma_0}^{alg}(T)=p^{m}T^g$, where $m=\mu_{E,\Sigma_0}^{alg}$. However, $\mu_{E,\Sigma_0}^{alg}=\mu_{E}^{alg}$ and the result follows.
\end{proof}

\section{Congruences over the Cyclotomic extension}\label{section3}
\par This section proves congruences for truncated $\Gamma$-Euler characteristics. For $i=1,2$, let $f_i$ be the eigencuspform associated to $E_i$. From here on in, let $\Sigma_0$ be the finite set of primes $l$ at which either $E_1$ or $E_2$ has bad reduction. Since both $E_1$ and $E_2$ have good reduction at $p$, the set $\Sigma_0$ does not contain $p$.
\begin{Lemma}\label{hypLemma}Let $i=1$ or $2$ and $l\in \Sigma_0$. The value $L_l(E_i,1)^{-1}$ is a $p$-adic unit if and only if $l+\beta_i(l)-a_l(f_i)$ is not divisible by $p$. As a result, $\Phi_{E_i}=1$ if and only if $l+\beta_i(l)-a_l(f_i)$ is not divisible by $p$ for all primes $l\in\Sigma_0$.
\end{Lemma}
\begin{proof}
Recall that the elliptic curves $E_1$ and $E_2$ have good ordinary reduction at $p$ and therefore $p\notin \Sigma_0$. Let $l\in \Sigma$, the value $L_l(E_i,1)^{-1}=l^{-1}(l+\beta_i(l)-a_l(f_i))$ is a $p$-adic unit if and only if $p$ does not divide $l+\beta_i(l)-a_l(f_i)$. The value $\Phi_{E_1}=\prod_{l\in \Sigma_0} |L_l(E_i,1)|_p=1$ if and only if $L_l(E_i,1)^{-1}$ is a $p$-adic unit for all $l\in \Sigma_0$.
\end{proof}\begin{Def}
For $i=1$ or $2$ we say that $E_i$ satisfies condition $(\star)$ if $l+\beta_i(l)-a_l(f_i)$ is not divisible by $p$ for all primes $l\in\Sigma_0$.
\end{Def}
We stress that our results are proved without stipulating condition $(\star)$ for $E_1$ or $E_2$, however, stipulating condition $(\star)$ simplifies the results. In our numerical examples we consider cases when $(\star)$ is satisfied and otherwise.
\begin{Th}\label{gammacongruence}
For elliptic curves $E_1$ and $E_2$ satisfying conditions $(\ref{one})$ to $(\ref{six})$ mentioned in the introduction, it is the case that $\Phi_{E_1}\times  \chi_t(\Gamma, E_1)=1$ if and only if $\Phi_{E_2}\times \chi_t(\Gamma,E_2)=1$.
If in addition condition $(\star)$ is satisfied for $E_1$ and $E_2$, then $ \chi_t(\Gamma, E_1)=1$ if and  only if $\chi_t(\Gamma,E_2)=1$.
\end{Th}

\begin{proof}
It follows from Corollory $\ref{prethmcorollory}$ that
\[\Phi_{E_i}\times  \chi_t(\Gamma, E_i)=1\Leftrightarrow \mu_{E_i}^{alg}=0\text{, and }\lambda_{E_i,\Sigma_0}^{alg}=g.\]It is a standard result that if $\mu_{E_1}^{alg}=0$ then $\mu_{E_2}^{alg}=0$. On the other hand, it is shown by Greenberg and Vatsal that if \[\mu_{E_1}^{alg}=\mu_{E_2}^{alg}=0\text{, then }\lambda_{E_1,\Sigma_0}^{alg}=\lambda_{E_2,\Sigma_0}^{alg}.\] This follows from \cite[Proposition 2.8]{greenbergvatsal} and the discussion on \cite[pp. 43]{greenbergvatsal} preceding Remark (2.10). Therefore, it follows that
\[\Phi_{E_1}\times  \chi_t(\Gamma, E_1)=1\Leftrightarrow \Phi_{E_2}\times \chi_t(\Gamma,E_2)=1.\]By Lemma $\ref{hypLemma}$, if condition $(\star)$ is satisfied then $\Phi_{E_1}=\Phi_{E_2}=1$. As a result, if in addition to conditions $(\ref{one})$ to $(\ref{six})$, condition $(\star)$ is satisfied, then 
\[\chi_t(\Gamma, E_1)=1\Leftrightarrow \chi_t(\Gamma,E_2)=1.\]
\end{proof}
The following corollary is a direct implication of the above theorem and the $p$-adic Birch and Swinnerton-Dyer formula (cf. Proposition $\ref{PerrinRiouresult}$). We note that $E_i(\Q)[p]=0$ for $i=1,2$ since the Galois representations on $E_i[p]$ are assumed to be irreducible.
\begin{Cor}\label{BSDmaincor}
Let $E_1$ and $E_2$ be elliptic curves satisfying conditions $(\ref{one})$ to $(\ref{six})$. Then \[\Phi_{E_1}p^{-g}\rho(E_1)=\Phi_{E_1}\times \frac{p^{-g}R_p(E_1/\Q)\times \#(\Sh (E_1/\Q)[p])}{\#(E_1(\Q)[p])^2}\times \tau(E_1)\times \#\tilde{E_1}(\F_p)^2\] is a $p$-adic unit if and only if
\[\Phi_{E_2}p^{-g}\rho(E_2)=\Phi_{E_2}\times \frac{p^{-g}R_p(E_2/\Q)\times \#(\Sh (E_2/\Q)[p])}{\#(E_2(\Q)[p])^2}\times \tau(E_2)\times \#\tilde{E_2}(\F_p)^2\]is a $p$-adic unit.
\end{Cor}
The following Proposition is a direct consequence of Lemma $\ref{truncatedECSelmer1}$. \begin{Prop}\label{TheoremchiLfunctions}
Let $E$ be an elliptic curve for which the conditions of Lemma $\ref{truncatedECSelmer1}$ are satisfied. If the main conjecture is satisfied for $E$, then $\chi_t(\Gamma,E)\sim \mathcal{L}^{(g)}(E/\Q,0)$.
\end{Prop}
There are analytic analogs of Theorem $\ref{gammacongruence}$ which we discuss. The following is a result of Greenberg and Vatsal.
\begin{Th}\cite[Theorem 3.10]{greenbergvatsal}\label{theoremLfunctions}
Let $E_1$ and $E_2$ elliptic curves satisfying the conditions $(\ref{one})$ to $(\ref{six})$. There exists a unit $w\in \Z_p^{\times}$ for which
\[\mathcal{L}^{\Sigma_0}(E_1/\Q,T)\equiv w \mathcal{L}^{\Sigma_0}(E_2/\Q,T)\mod{p \Lambda}.\]
\end{Th}

Vatsal (cf. \cite[Corollary 1.11]{VatsalCP}) also deduces congruences for special values of complex L-functions attached to eigenforms whose Fourier coefficients are congruent. Such congruences are interpolated by congruences of $p$-adic L-functions. The following is the analytic version of Theorem $\ref{gammacongruence}$.
\begin{Th}\label{corolloryLfunctions}
There exists a unit $u\in \Z_p^{\times}$ such that there is a congruence
\[\Phi_{E_1}\cdot \mathcal{L}^{(g)}(E_1/\Q,0)\equiv u \Phi_{E_2}\cdot \mathcal{L}^{(g)}(E_2/\Q,0)\mod{p \Z_p}.\]
\end{Th}
\begin{proof}
As mentioned previously, under the assumptions on $E_1$ and $E_2$, the order of the zero of $f_{E_i}^{alg}(T)$ is equal to $g$ (cf. \cite[Theorem 2]{heightpairings2}). It is a well known result of Kato that $f_{E_i}^{alg}(T)$ divides $f_{E_i}^{anal}(T)$. This implies that $T^g$ divides $\mathcal{L}(E_i/\Q,T)$. Find that
\[\begin{split}\frac{1}{g!}\left(\frac{d}{dT}\right)^g\mathcal{L}^{\Sigma}(E_i/\Q, T)_{|T=0}=&\mathcal{L}^{(g)}(E_i/\Q,0)\prod_{l\in \Sigma} \mathcal{P}_l(E_i/\Q,0)\\=&\mathcal{L}^{(g)}(E_i/\Q,0)\prod_{l\in \Sigma} L_l(E_i,1)^{-1}.
\end{split}\] The assertion follows from Theorem $\ref{theoremLfunctions}$.
\end{proof}

\par Let $\Gamma':=\Gal(\K/\Q)\simeq \Z_p^{\times}$ and $\Delta:=\ker\left(\Gamma'\rightarrow \Gamma\right)$ so that $\Delta$ can be identified with $\mu_{p-1}\subset \Z_p^{\times}$. For $i=1,2$ as before, set $\chi_t(\Gamma', E_i):=\chi_t(\Gamma', \Sel(E_i/\K))$. Let $E$ be an elliptic curve with good ordinary reduction at $p$ and $\tilde{E}$ be the reduced curve at $p$. The extensions $\Q^{cyc}$ and $\Q(\mu_{p^{\infty}})$ are deeply ramified extensions (see \cite{CG}). Suppose $F_{\infty}/\Q$ is a deeply ramified extension. Let $l$ be a prime and $w|l$ a prime of $F_{\infty}$ above $l$. If $l\neq p$, there is a canonical isomorphism 
    \begin{equation}\label{lequalsp}H^1(F_{\infty,w}, E)_{p^{\infty}}\simeq H^1(F_{\infty,w}, E_{p^{\infty}}).\end{equation}This follows from a standard Kummer theory argument. If $l=p$,
    \begin{equation}\label{pequalsp}H^1(F_{\infty,w}, E)_{p^{\infty}}\simeq H^1(F_{\infty,w}, \tilde{E}_{p^{\infty}}),\end{equation}which is proved for instance in \cite[Proposition 4.8]{CG}.
\begin{Lemma}\label{Selmerequality}
Let $E$ be an elliptic curve over $\Q$ with good ordinary reduction at $p$. There is a canonical isomorphism induced by restriction
\[\Sel(E/\Q^{\text{cyc}})\simeq \Sel(E/\K)^{\Delta}.\]
\end{Lemma}
\begin{proof}
Consider the following diagram
\begin{equation}\label{diagram 1}
\begin{tikzcd}[column sep = small, row sep = large]
0\arrow{r} & \Sel(E/\Q^{cyc}) \arrow{r}\arrow{d}{f} & H^1(\Q^{cyc},E_{p^{\infty}})\arrow{r} \arrow{d}{g} & \bigoplus_{w} H^1(\Q^{cyc}_w,E)_{p^{\infty}}\arrow{r} \arrow{d}{h} & 0\\
0\arrow{r} & \Sel(E/\K)^{\Delta} \arrow{r} & H^1(\K,E_{p^{\infty}})^{\Delta} \arrow{r}  &\bigoplus_{w} H^1(\K_w,E)_{p^{\infty}}.
\end{tikzcd}
\end{equation}
We show that $g$ is an isomorphism and that $h$ is injective. An application of the Snake lemma implies that $f$ is an isomorphism. First let us show that $g$ is an isomorphism. By inflation-restriction, \[\ker g=H^1(\Delta, E(\K)_{p^{\infty}}).\] Since the order of $\Delta$ is coprime to $p$, $\ker g=0$. Likewise, \[H^2(\Delta, E(\K)_{p^{\infty}})=0\] and it follows that $\operatorname{cok} g=0$ and therefore $g$ is an isomorphism. The map 
\[h_w:H^1(\Q^{cyc}_w,E)_{p^{\infty}}\rightarrow H^1(\K_w,E)_{p^{\infty}}\]
can be identified with the restriction map
\[h_w:H^1(\Q^{cyc}_w,D)\rightarrow H^1(\K_w,D)\] where $D=E_{p^{\infty}}$ if $w$ does not divide $p$ and $D=\tilde{E}_{p^{\infty}}$ if $w$ divides $p$. Let $\Delta_w:=\Gal(\K_w/\Q^{cyc}_w)$, the order of $\Delta_w$ is coprime to $p$. This restriction map fits into the inflation-restriction sequence
\[H^1(\Delta_w, D(\K))\rightarrow H^1(\Q^{cyc}_w,D)\rightarrow H^1(\K_w,D)^{\Delta_w}.\]
Since $\Delta_w$ has order coprime to $p$, it follows that $H^1(\Delta_w, D(\K))=0$ and therefore $h_w$ is injective. This completes the proof of the Lemma. 
\end{proof}
\begin{Prop}\label{comparisonprop}
Let $E$ be an elliptic curve with good ordinary reduction at $p$. The truncated Euler characteristic $\chi_t(\Gamma, E)$ is finite if and only if $\chi_t(\Gamma', E)$ is finite. If this is the case, then $\chi_t(\Gamma, E)=\chi_t(\Gamma', E)$.
\end{Prop}
\begin{proof}
Let \[\psi:H^0(\Gamma', \Sel(E/\K))\rightarrow H^1(\Gamma', \Sel(E/\K))\] the composite of the map \[\varphi: H^0(\Gamma, \Sel(E/\K)^{\Delta})\rightarrow H^1(\Gamma, \Sel(E/\K)^{\Delta})\] with the inflation map \[\operatorname{inf}:H^1(\Gamma,\Sel(E/\K)^{\Delta})\rightarrow H^1(\Gamma', \Sel(E/\K)).\] The inflation map fits the inflation-restriction sequence
\[0\rightarrow H^1(\Gamma,\Sel(E/\K)^{\Delta})\rightarrow H^1(\Gamma', \Sel(E/\K))\rightarrow H^1(\Delta, \Sel(E/\K))^{\Gamma}.\]Since the order of $\Delta$ is coprime to $p$, it follows that \[H^1(\Delta, \Sel(E/\K))=0\]and therefore the inflation map is an isomorphism. Therefore, $\psi$ can be identified with $\varphi$. By Lemma $\ref{Selmerequality}$, the map $\varphi$ can be identified with the map 
\[\varphi_E:H^0(\Gamma, \Sel(E/\Q^{cyc}))\rightarrow H^1(\Gamma, \Sel(E/\Q^{cyc})).\]Putting it all together, \[\chi_t(\Gamma', E)=\frac{\#\ker \psi}{\#\operatorname{cok} \psi}=\frac{\#\ker \varphi}{\#\operatorname{cok} \varphi}=\frac{\#\ker \varphi_E}{\#\operatorname{cok} \varphi_E}=\chi_t(\Gamma, E).\]
\end{proof}
Let $\Gamma'=\Gal(\K/\Q)$, the following congruence is a consequence of Theorem $\ref{gammacongruence}$ and Proposition $\ref{comparisonprop}$.
\begin{Cor}\label{cyclotomiccongruence} 
For elliptic curves $E_1$ and $E_2$ satisfying conditions $(\ref{one})$ to $(\ref{six})$, it is the case that $\Phi_{E_1}\times  \chi_t(\Gamma', E_1)=1$ if and only if $\Phi_{E_2}\times \chi_t(\Gamma',E_2)=1$.
If in addition condition $(\star)$ is satisfied for $E_1$ and $E_2$, then $ \chi_t(\Gamma', E_1)=1$ if and  only if $\chi_t(\Gamma',E_2)=1$.
\end{Cor}
\section{Congruences of $G$-Euler Characteristics}
\par In this section, we prove congruences over $G$ where $G=\Gal(\Q_{\infty}/\Q)$ is the Galois group of an admissible extension. This is achieved by relating the truncated Euler characteristic over $G$ to that over the cyclotomic extension. The extensions considered will be the false Tate curve extension and the admissible $p$-adic Lie-extension arising from the $p^{\infty}$-torsion points of the elliptic curves $E_i$, for $i=1,2$. In the false Tate curve case, the corresponding Galois group is a semidirect product $\Z_p^{\times} \ltimes \Z_p$ and in the other case it is a finite index subgroup of $\GL_2(\Z_p)$.

\subsection{Congruences for the False Tate Curve Extension}
\par Let $m$ be a positive integer coprime to $Np$, where we recall that $N$ is the product of the conductors of $E_1$ and $E_2$. Let $\Q_{\infty}$ be the false Tate curve extension $\Q_{\infty}:=\Q(\mu_{p^{\infty}}, m^{\frac{1}{p^{\infty}}})$. Our assumptions imply that the reduction type of $E$ does not change in any number field extension of $\Q$ contained in $\Q_{\infty}$. In particular, if $w$ is a prime above $l$ in $\Q(\mu_p)$ then $E$ has the same reduction type at $l$ as well as $w$. Recall that $G:=\Gal(\Q_{\infty}/\Q)$ and $H:=\Gal(\Q_{\infty}/\Q^{\text{cyc}})$. Let $\Gamma'=\Gal(\K/\Q)$ and identify $G/H\simeq \Gamma$. Let $\Sigma$ be a set finite set of primes containing all primes $l|mp$ and all primes $l$ at which either $E_1$ or $E_2$ has bad reduction. 
\begin{Def}\cite[Definition 2.4]{shekharsujatha}
Let $\mathfrak{M}(E)$ denote the subset of primes $l\in \Sigma$ consisting of the primes $l$ dividing $m$ such that $E$ has good reduction at $l$ and the corresponding reduced curve has a
point of order $p$.
\end{Def}
Let $P_0$ be the set of primes $l\neq p$ which are ramified in $\Q_{\infty}$, this is the set of primes $l|m$. Recall that it is stipulated that $E$ has good reduction at each prime $l|m$ and as a consequence the set of primes $\mathfrak{M}(E)\subset P_0$ coincides with the set of primes $\Sigma_3(E)$ defined in \cite{shekharsujatha}.
\begin{Th}\cite[Theorem 4.10]{HV}\label{GtoGamma}
\[\chi_t(G,E)\sim\chi_t(\Gamma',E)\times \prod_{l\in \mathfrak{M}(E)} L_l(E,1)^{-1}.\]
\end{Th}

\begin{Th}\label{Gcongruence}
For elliptic curves $E_1$ and $E_2$ satisfying conditions $(\ref{one})$ to $(\ref{six})$ as in the introduction. It is the case that $\Phi_{E_1}\times  \chi_t(G, E_1)=1$ if and only if $\Phi_{E_2}\times \chi_t(G,E_2)=1$.
If in addition condition $(\star)$ is satisfied for $E_1$ and $E_2$, then $ \chi_t(G, E_1)=1$ if and  only if $\chi_t(G,E_2)=1$.
\end{Th}
\begin{proof}
\par Let $E$ denote any one of the elliptic curves $E_1$ or $E_2$. Recall that $m$ is by assumption, coprime to $pN_E$ and hence $E$ has good reduction at each prime $l|m$. By \cite[Lemma 2.5]{shekharsujatha}, the set of primes $\mathfrak{M}(E)$ is the set of primes $l|m$ at which $p|L_l(E,1)^{-1}$. At each prime $l|m$, $E_1$ and $E_2$ both have good reduction, therefore, $a_l(f_1)\equiv a_l(f_2)\mod{p}$ and therefore,
\[L_l(E_1,1)^{-1}\equiv L_l(E_2, 1)^{-1}\mod{p}.\]As a consequence, $\mathfrak{M}(E_1)=\mathfrak{M}(E_2)$. Let $l\in \mathfrak{M}(E_1)=\mathfrak{M}(E_2)$.
By Theorem $\ref{GtoGamma}$, 
\[\chi_t(G,E_i)\sim\chi_t(\Gamma',E_i)\times \prod_{l\in \mathfrak{M}(E_i)} L_l(E_i,1)^{-1}.\]
On the other hand, by Corollary $\ref{cyclotomiccongruence}$,
\[\Phi_{E_1}\times \chi_t(\Gamma',E_1)=1\Leftrightarrow  \Phi_{E_2}\times\chi_t(\Gamma', E_2)=1.\]Putting it all together, one obtains the congruence 
\[\Phi_{E_1}\times\chi_t(G,E_1)=1\Leftrightarrow  \Phi_{E_2}\times \chi_t(G, E_2)=1.\]By Lemma $\ref{hypLemma}$, if condition $(\star)$ is satisfied then $\Phi_{E_1}=\Phi_{E_2}=1$. The assertion of the Theorem follows.
\end{proof}
\subsection{Congruences for $\GL_2$-extensions}
\par Let $E$ be an elliptic curve over $\Q$ with good ordinary reduction at $p$ for which the Galois module $E[p]$ is irreducible. Denote by $\Q_{\infty}(E):=\Q(E_{p^{\infty}})$ be the field obtained by adjoining the coordinates of the group of $p$-power division points of $E$. Let $G_E:=\Gal(\Q_{\infty}(E)/\Q)$ and identify $G_E$ with its image in $\operatorname{Aut}(T_p(E))\simeq \GL_2(\Z_p)$, where as usual $T_p(E):=\varprojlim_n E_{p^n}$ is the $p$-adic Tate-module of $E$. By a well-known theorem of Serre, $G_E$ is an open subgroup of $\GL_2(\Z_p)$. Let $H_E:=\Gal(\Q_{\infty}(E)/\Q^{cyc})$ and as usual, $\Gamma=G_E/H_E$. It was shown by Kato \cite{kato} that $\Sel(E/\Q^{cyc})$ is a cotorsion $\Lambda(\Gamma)$-module. It follows from this that the Selmer group $\Sel(E/\Q_{\infty}(E))$ is a cotorsion $\Lambda(G_E)$-module (cf. \cite{CH}). Let $\mathcal{M}_E$ be the set of primes at which the $j$-invariant of $E$ is non-integral. 

\begin{Th}\label{GL2toGamma}\cite[Theorem 3.1]{CSSLinks}
Suppose that $\chi_t(\Gamma, E)$ is finite then so is $\chi_t(G_E, E)$ and we have
\[\chi_t(G_E, E)\sim \chi_t(\Gamma, E)\times \prod_{l\in \mathcal{M}_E} L_l(E,1)^{-1}.\]
\end{Th}We let $G_i:=G_{E_i}$.
\begin{Th}\label{GEcongruence}
Let $E_1$ and $E_2$ be elliptic curves satisfying conditions $(\ref{one})$ to $(\ref{six})$ as in the introduction. And further assume that at each prime $l$ where $E_i$ does not have potentially good reduction, \[p\nmid l-a_l(E_i).\] Then, it is the case that $\Phi_{E_1}\times  \chi_t(G_1, E_1)=1$ if and only if $\Phi_{E_2}\times \chi_t(G_2,E_2)=1$.
If in addition condition $(\star)$ is satisfied for $E_1$ and $E_2$, then $ \chi_t(G_1, E_1)=1$ if and  only if $\chi_t(G_2,E_2)=1$.
\end{Th}

\begin{proof}
By Theorem $\ref{gammacongruence}$, \[\Phi_{E_1}\cdot\chi_t(\Gamma , E_1)=1\Leftrightarrow  \Phi_{E_2}\cdot\chi_t(\Gamma ,E_2)=1.\] By Theorem $\ref{GL2toGamma}$, 
\[\chi_t(G_i, E_i)\sim \chi_t(\Gamma, E_i)\times \prod_{l\in \mathcal{M}_{E_i}} L_l(E_i,1)^{-1}.\] By \cite[Proposition 5.5]{Silverman}, a prime $l\in \mathcal{M}_{E_i}$ if and only if $E_i$ does not have potentially good reduction at $l$. By our assumption, at every prime $l\in \mathcal{M}_{E_i}$, we have that $L_l(E_i,1)$ is a $p$-adic unit. The assertion follows from this.
\end{proof}
\section{Numerical Examples}\label{section5}
\par In this short section we discuss some concrete examples which illustrate our results for $p=5$. All our computations are aided by Sage.
\subsection{Example 1:}
\par Let $E_1=201c1$ and $E_2=469a1$. Both elliptic curves have rank $1$ and good ordinary reduction at the prime $5$. Further, there is an isomorphism of the residual Galois representations $E_1[5]\simeq E_2[5]$, at the prime $5$, and the conditions $(\ref{one})$ to $(\ref{six})$ are satisfied. The example shows that $\chi_t(\Gamma,E_i)=1$ and $\Phi_{E_i}=1$ for $i=1,2$. As a result, 
\[\Phi_{E_1}\times \chi_t(\Gamma,E_1)=1\text{, and }\Phi_{E_2}\times \chi_t(\Gamma,E_2)=1.\] This verifies Theorem $\ref{gammacongruence}$. We further discuss results for more general $5$-adic Lie-extensions as in Theorem $\ref{Gcongruence}$ and Theorem $\ref{GEcongruence}$.
\par The elliptic curve $E_1$ has bad reduction at $3$ and $67$ and $E_2$ has bad reduction at $7$ and $67$. Also, note that \[
\begin{split}
    &\beta_1(3)=0, \beta_1(7)=1, \beta_1(67)=0,\\
    &\beta_2(3)=1,\beta_2(7)=0,\beta_2(67)=0
\end{split}\]and
\[
\begin{split}
    &3+\beta_1(3)-a_3(f_1)=4, 7+\beta_1(7)-a_7(f_1)=11, 67+\beta_1(67)-a_{67}(f_1)=68,\\
     &3+\beta_1(3)-a_3(f_1)=3, 7+\beta_1(7)-a_7(f_1)=8, 67+\beta_1(67)-a_{67}(f_1)=68.\\
\end{split}\]
Therefore for $E_1$ and $E_2$, condition $(\star)$ is satisfied and as a consequence, \[\Phi_{E_1}=\Phi_{E_2}=1.\]  By Theorem $\ref{gammacongruence}$, \[\chi_t(\Gamma, E_1)=1\Leftrightarrow \chi_t(\Gamma,E_2)=1.\] The $5$-adic regulators \[R_5(E_1/\Q)\sim 5\text{ and }R_5(E_2/\Q)\sim 5.\] For $i=1,2$, \[\#(\Sh(E_i/\Q)[5])=1,\tau_5(E_i)=1, \#\tilde{E}_i(\F_5)=1\text{ and } \#E_i(\Q)[5]=1.\] {By Proposition $\ref{PerrinRiouresult}$, \[\chi_t(\Gamma, E_i)=1\] for $i=1,2$, thus illustrating Theorem $\ref{gammacongruence}$.} \par Let us consider congruences over the false Tate curve extension.
For any integer $m$ not divisible by $3,5,7$ and $67$ let $G$ be the Galois group of the false Tate curve extension\[G:=\Gal(\Q(\mu_{5^{\infty}}, m^{\frac{1}{5^{\infty}}})/\Q).\]
Since condition $(\star)$ is satisfied, by Theorem $\ref{Gcongruence}$,
    \[\chi_t(G,E_1)=1\Leftrightarrow  \chi_t(G,E_2)=1.\]
\par Let us consider congruences over extensions cut out by $5^{\infty}$-torsion points of $E_1$ and $E_2$. For $i=1,2$, all primes $l$ of bad reduction, $l-a_l(f_i)$ is not divisible by $5$. Since condition $(\star)$ is satisfied, by Theorem $\ref{GEcongruence}$, \[\chi_t(G_1,E_1)=1\Leftrightarrow  \chi_t(G_2,E_2)=1.\] By Theorem $\ref{GL2toGamma}$, \[\chi_t(G_i, E_i)\sim \chi_t(\Gamma, E_i)\times \prod_{l\in \mathcal{M}_{E_i}} L_l(E_i,1)^{-1}.\] At each prime $l\in \mathcal{M}_{E_i}$, $l-a_l(f_i)$ is not divisible by $5$ and therefore $L_l(E_i,1)=1-a_l(f_i)l^{-1}$ is a $5$-adic unit. Therefore $\chi_t(G_i,E_i)=\chi_t(\Gamma, E_i)=1$ for $i=1,2$.
\subsection{Example 2:}
{This example illustrates the role played by the factors $\Phi_{E_i}$ for $i=1,2$.} In this example, condition $(\star)$ is not satisfied. Let $E_1$ and $E_2$ be the rank $1$ elliptic curves $E_1=37a1$ and $E_2=1406g1$. Both curves satisfy conditions $(\ref{one})$ to $(\ref{six})$. In this example \[\chi_t(\Gamma,E_1)=1\text{, and }\chi_t(\Gamma,E_2)=5^2,\]in particular, the truncated Euler-characteristics for $E_1$ and $E_2$ are different. On the other hand, it will be shown that 
\[\Phi_{E_1}=5^2\text{, and }\Phi_{E_2}=1.\] As a result, 
\[\Phi_{E_1}\times \chi_t(\Gamma,E_1)=5^2\text{, and }\Phi_{E_2}\times \chi_t(\Gamma,E_2)=5^2\] are both divisible by $5$, as predicted by Theorem $\ref{gammacongruence}$. We further discuss results for more general $5$-adic Lie-extensions as in Theorem $\ref{Gcongruence}$ and Theorem $\ref{GEcongruence}$. \par The curve $E_2$ has bad reduction at $2,19$ and $37$. Find that,
\[
\begin{split}
    &2+\beta_1(2)-a_2(f_1)=5, 19+\beta_1(19)-a_{19}(f_1)=20, 37+\beta_1(37)-a_{37}(f_1)=38,\\
    &2+\beta_1(2)-a_2(f_2)=1, 19+\beta_1(19)-a_{19}(f_2)=18, 37+\beta_1(37)-a_{37}(f_2)=38.\\
\end{split}\]
 It may be found that \[\Phi_{E_1}=5^2\text{ and } \Phi_{E_2}=1.\] The $5$-adic regulator \[R_5(E_1/\Q)\sim 5\] and it may be shown that \[\#(\Sh(E_1/\Q)[5])=1, \tau(E_1)=1, \#\tilde{E}_1(\F_5)=1\text{ and } \#E_1(\Q)[5]=1.\]On the other hand for $E_2$, the $5$-adic regulator \[R_5(E_2/\Q)\sim5\] and \[\#(\Sh(E_2/\Q)[5])=1, \tau(E_2)=150, \#\tilde{E}_2(\F_5)=1 \text{ and }\#E_2(\Q)[5]=1.\]Our computations show that \[\chi_t(\Gamma,E_1)=1 \text{ and } \chi_t(\Gamma, E_2)=5^2.\]
 Recall that $\Phi_{E_1}=5^2$ and $\Phi_{E_2}=1$ and Theorem $\ref{gammacongruence}$ asserts that 
 \[5|(\Phi_{E_1}\times \chi_t(\Gamma,E_1))\Leftrightarrow 5|(\Phi_{E_2}\times  \chi_t(\Gamma,E_2)).\]
\par We now illustrate our congruence results for false Tate curve extensions. For any integer $m$ not divisible by $2,5,19$ and $37$ and $G$ the Galois group of the false Tate curve extension $G:=\Gal(\Q(\mu_{5^{\infty}}, m^{\frac{1}{5^{\infty}}})/\Q)$. By Theorem $\ref{Gcongruence}$, $5^2\chi_t(G,E_1)\equiv \chi_t(G,E_2)\mod{5\Z_5}$. Therefore, $5$ divides $\chi_t(G,E_2)$ for any such $m$. \par Next, we compute truncated Euler characteristics for extensions cut out by $5^{\infty}$-torsion for $E_i$ for $i=1,2$. Even though condition $(\star)$ is not satisfied for $E_1$, Theorem $\ref{GEcongruence}$ applies. In fact, at every prime $l$ at which $E_i$ has bad reduction, $L_l(E_i,1)$ is a $5$-adic unit. By Theorem $\ref{GL2toGamma}$, $\chi(G,E_i)=\chi(\Gamma, E_i)$ for $i=1,2$. Therefore, \[\chi(G, E_1)=1\text{ and } \chi(G,E_2)=5^2.\]
\subsection{Example 3:} 
\par Let $E_1=82a1$ and $E_2=902a1$, the elliptic curves satisfy conditions $\ref{one}$ to $\ref{six}$ in the introduction. In particular, $E_1[5]\simeq E_2[5]$ as Galois-modules. It may be verified that condition $(\star)$ is satisfied, and as a result,
\[\Phi_{E_1}=\Phi_{E_2}=1.\]Repeating calculations from previous examples, find that\[\chi_t(\Gamma,E_1)=5\text{, and }\chi_t(\Gamma,E_2)=5^3.\]This implies that
\[\Phi_{E_1}\times \chi_t(\Gamma,E_1)=5 \text{, and }\Phi_{E_1}\times \chi_t(\Gamma,E_1)=5^3.\]These values are both divisible by $5$, as predicted by Theorem $\ref{gammacongruence}$, however, they are different. Hence the truncated Euler characteristic is not determined by the residual representation, only determined up to congruence.
\par We now consider congruences over false Tate curve extensions. The elliptic curve $E_1$ has bad reduction at $2$ and $41$, $E_2$ on the other hand has bad reduction at $2$,$11$ and $41$. For any integer $m$ not divisible by $2,11$ and $41$ and $G$ the Galois group of the false Tate curve extension $G:=\Gal(\Q(\mu_{5^{\infty}}, m^{\frac{1}{5^{\infty}}})/\Q)$. By Theorem $\ref{Gcongruence}$, \[\chi_t(G,E_1)\equiv \chi_t(G,E_2)\mod{5\Z_5}.\] Therefore, $5$ divides $\chi_t(G,E_2)$ for any such $m$. \par Next, we compute truncated Euler characteristics for extensions cut out by $5^{\infty}$-torsion for $E_i$ for $i=1,2$. Since condition $(\star)$ is satisfied, it follows that at every prime $l$ at which $E_i$ has bad reduction, $L_l(E_i,1)$ is a $5$-adic unit. By Theorem $\ref{GL2toGamma}$, $\chi(G,E_i)=\chi(\Gamma, E_i)$ for $i=1,2$. Therefore, \[\chi(G, E_1)=5\text{ and } \chi(G,E_2)=5^3.\]


\begin{thebibliography}{1}
\bibitem{padicBSD}Balakrishnan, Jennifer, J. Müller, and William Stein. "A $p$-adic analogue of the conjecture of Birch and Swinnerton-Dyer for modular abelian varieties." Mathematics of Computation 85.298 (2016): 983-1016.
\bibitem{CG}Coates, John, and Ralph Greenberg, \textit{Kummer theory for abelian varieties over local fields}, Inventiones mathematicae 124.1 (1996): 129-174.
\bibitem{CH}Coates, J. H., and Susan Howson, \textit{Euler characteristics and elliptic curves II}, Journal of the Mathematical Society of Japan 53.1 (2001): 175-235.

\bibitem{CSSLinks}Coates, John, et al, \textit{Links between cyclotomic and $\text{GL}_2$ Iwasawa theory}, Doc. Math 2003 (2003): 187-215.
\bibitem{GCEC}Coates, John, and Ramdorai Sujatha \textit{Galois cohomology of elliptic curves}, Narosa, 2000.
\bibitem{Dok}Dokchitser, Tim, et al, \textit{Computations in non-commutative Iwasawa theory}, Proceedings of the London Mathematical Society 94.1 (2007): 211-272.
\bibitem{EPW}Emerton, Matthew, Robert Pollack, and Tom Weston, \textit{Variation of Iwasawa invariants in Hida families}, Inventiones mathematicae 163.3 (2006): 523-580.
\bibitem{greenbergvatsal}Greenberg, Ralph, and Vinayak Vatsal, \textit{On the Iwasawa invariants of elliptic curves}, Inventiones mathematicae 142.1 (2000): 17-63.
\bibitem{HV}Hachimori, Yoshitaka, and Otmar Venjakob, \textit{Completely faithful Selmer groups over Kummer extensions}, Documenta Math., Extra Volume: Kazuya Kato's Fiftieth Birthday (2003): 443-478.
\bibitem{kato}Kato, Kazuya, \textit{p-adic Hodge theory and values of zeta functions of modular forms}, Ast\'erisque 295 (2004): 117-290.

\bibitem{MTT}Mazur, Barry, John Tate, and Jeremy Teitelbaum, \textit{On p-adic analogues of the conjectures of Birch and Swinnerton-Dyer}, Inventiones mathematicae 84.1 (1986): 1-48.
\bibitem{perrinriou}Perrin-Riou, Bernadette, \textit{Th\'eorie d'Iwasawa et hauteurs p-adiques}, Inventiones mathematicae 109.1 (1992): 137-185.

\bibitem{Silverman}Silverman, Joseph H, \textit{The arithmetic of elliptic curves}, Vol. 106.
Springer Science and Business Media, 2009.
\bibitem{heightpairings1}Schneider, Peter, \textit{p-adic height pairings I}, Inventiones mathematicae 69.3 (1982): 401-409.

\bibitem{heightpairings2}
Schneider, Peter, \textit{p-adic height pairings. II}, Inventiones mathematicae 79.2 (1985): 329-374.
\bibitem{shekharsujatha}Shekhar, Sudhanshu, and R. Sujatha, \textit{Euler characteristic and congruences of elliptic curves}, M\"unster J. Math. 7.1 (2014): 327-343.

\bibitem{skinnerurban}Skinner, Christopher, and Eric Urban. \textit{The Iwasawa main conjectures for }$\GL_2$. Inventiones mathematicae 195.1 (2014): 1-277.

\bibitem{VatsalCP}Vatsal, Vinayak, \textit{Canonical periods and congruence formulae}, Duke mathematical journal 98.2 (1999): 397.

\bibitem{Zerbes}Zerbes, Sarah Livia, \textit{Generalised Euler characteristics of Selmer groups}, Proceedings of the London Mathematical Society 98.3 (2008): 775-796.
\end{thebibliography}
\end{document}